\documentclass[a4paper,12pt]{amsart}

\usepackage[all]{xy}
\usepackage{graphics}
\usepackage{color}
\usepackage[T1]{fontenc}

\usepackage[upright]{fourier}
\usepackage{kerkis}
\usepackage{eucal}
\usepackage{parskip}
\newcommand{\B}[1]{{\mathbf #1}}

\newcommand{\F}[1]{{\mathfrak #1}}

\usepackage[colorlinks]{hyperref}

\newtheorem{theorem}[equation]{Theorem}
\newtheorem{corollary}[equation]{Corollary}
\newtheorem{lemma}[equation]{Lemma}
\newtheorem{proposition}[equation]{Proposition}
\theoremstyle{definition}
\newtheorem{example}[equation]{Example}
\theoremstyle{remark}

\newtheorem{remark}[equation]{Remark}

\numberwithin{equation}{section}
\numberwithin{figure}{section}
\numberwithin{table}{section}

\newcommand\OP{\operatorname}
\newcommand\Diff{\OP{Diff}}
\newcommand\Homeo{\OP{Homeo}}
\newcommand\Hom{\OP{Hom}}

\title{On distortion in groups of homeomorphisms}
\author{\'Swiatos\l aw R. Gal}
\address{Uniwersytet Wroc\l awski \& Uniwersit\"at Wien}
\email{sgal@math.uni.wroc.pl}
\author{Jarek K\k edra}
\address{University of Aberdeen \& Uniwersytet Szcze\-ci\'n\-ski}
\email{kedra@abdn.ac.uk}


\begin{document}

\begin{abstract}
Let $X$ be a path-connected topological space admitting a universal cover.
Let $\Homeo(X,\F a)$ denote the group of homeomorphisms of $X$
preserving degree one cohomology class $\F a$.

We investigate the distortion in $\Homeo(X,\F a)$.  Let $g$ be an element of
$\Homeo(X,\F a)$. We define a Nielsen-type equivalence relation on the space
of $g$-invariant Borel probability measures on $X$ and prove that if
a homeomorphism $g$ admits two non\-e\-qui\-valent invariant measures then it
is undistorted.
We also define a local rotation number of a homeomorphism generalising
the notion of the rotation of a homeomorphism of the circle. Then we prove
that a homeomorphism is undistorted if its rotation number is nonconstant.
\end{abstract}

\maketitle

\section{Introduction and the statement of the results}
\label{S:intro}

Let $X$ be a path-connected topological space admitting a universal
cover. Let $\Homeo(X,\F a)$ denote the group of homeomorphisms of $X$
preserving a cohomology class $\F a\in H^1(X;\B R)$.

In the present
paper, we study distortion in $\Homeo(X,\F a)$.
We define the distortion in groups and briefly discuss this notion
in Section~\ref{SS:distortion}. 

Let $g\in\Homeo(X,\F a)$.
By $\alpha$ we denote a singular one-cocycle representing the class $\F a$.
Let  $\F K_\alpha(g)\colon X\to\B R$ be a function such that
$$
\delta\F K_\alpha(g)=g^*\alpha-\alpha,
$$
where $\delta$ denotes the codifferential map $C^0(X;\B R)\to C^1(X,\B R)$
between singular zero-cocycles (functions) and one-cocycles on $X$.
Notice that $\F K_\alpha(g)$ exists since $g^*\alpha$ and $\alpha$ are in the
same cohomology class.
In other terms a function $\F K_\alpha(g)$ is defined by the condition that
for any $x,y\in X$ one has
$\F K_\alpha(g)(y)-\F K_\alpha(g)(x)=\int_\gamma\left(g^*\alpha-\alpha\right)$,
where $\gamma$ denotes any path between $x$ and $y$, and
the expression $\int_\gamma\sigma$
denotes the natural pairing of a chain $\gamma$
and a cochain $\sigma$. 
We find this nonstandard notation useful because we would like to think
that the cocycle $\alpha$ is defined by the integration of a 
differential form over smooth paths.

We discuss the function $\F K_\alpha(g)$ in Section \ref{SS:one-cocycle}
where we prove that it is continuous for a suitable choice of the cocycle $\alpha$.
In particular, the function $\F K_\alpha(g)$ is integrable with respect to any
Borel probability measure on $X$. 
If not stated otherwise, we assume that $\alpha$ is chosen in such a way.

Assume that the homeomorphism $g\in \Homeo(X,\F a)$ admits two invariant Borel
probability measures $\mu$ and $\nu$.
We say that $\mu$ and $\nu$ are {\bf $\F a$-Nielsen equivalent} (we will motivate
the name after the proof of Corollary \ref{C:distortion}) if
$$
\int\F K_\alpha(g)\mu= \int\F K_\alpha(g)\nu.
$$
It is clear that $\F K_\alpha(g)$ is defined up to an additive constant, but the difference
$\int\F K_\alpha(g)\mu-\int\F K_\alpha(g)\nu$ does not deppend on this choice.
If the class $\F a$ is fixed we will simply call measures Nielsen equivalent.

The following theorem is proven in Section~\ref{SS:proof1}.

\begin{theorem}\label{T:distortion}
Let $X$ be compact and let $g\in \Homeo(X,\F a)$ 
Assume that $g$ admits invariant Nielsen nonequivalent measures
then $g$ is undistorted in $\Homeo(X,\F a)$.
\end{theorem}

The first corollary of the above theorem is when the measures are supported on
two fixed points of $g$.

\begin{corollary}\label{C:distortion}
Let $X$ be compact and let $g\in \Homeo(X,\F a)$ 
where $\F a\in H^1(X;\B R)$ is represented by a one-cocycle $\alpha$. 
Suppose that $g$ has two fixed points $x,y\in X$ and let $\gamma$ 
be a path from $x$ to $y$. If 
$$
\langle\F a,g\gamma-\gamma\rangle\neq 0
$$
then $g$ is undistorted in $\Homeo(X,\F a)$.
\end{corollary}

\begin{proof}
We have to check that atomic measures $\delta_x$ and $\delta_y$ supported on $x$ and $y$
are not Nielsen equivalent.
By the definition of $\F K_\alpha(g)$ we have that
$$
\int\F K_\alpha(g)(\delta_y-\delta_x)=\F K_\alpha(g)(y)-\F K_\alpha(g)(x)=\int_\gamma g^*\alpha-\alpha=\langle\F a,g\gamma-\gamma\rangle\neq0.
$$
\end{proof}

Two fixed points $x,y$ of a map $g\colon X\to X$
are called {\bf Nielsen equivalent} if there exists a path
$\gamma$ from $x$ to $y$ such that $\gamma$ and $g\gamma$
are homotopic modulo the endpoints. The hypothesis
of the above theorem implies that the homeomorphism
$g$ has two fixed points which are Nielsen nonequivalent
in a stronger sense. Namely, the cycle $g\gamma - \gamma$
is homologically nontrivial.

\begin{example}\label{E:surface_nielsen}
Let $G\subset \Homeo(\Sigma)$ be a group of homeomorphisms
of a closed oriented surface $\Sigma$ acting trivially
on the first cohomology of~$\Sigma$. Suppose that  $g\in G$ has
two fixed points $x,y\in \Sigma$ such that
$g\gamma - \gamma$ is a homologically nontrivial loop,
where $\gamma$ is a path from $x$ to $y$. 
Then $g$ is undistorted in $G$. Indeed, there exists a
cohomology class $\F a\in H^1(\Sigma;\B Z)$ evaluating
nontrivially on $g\gamma - \gamma$.
\hfill $\diamondsuit$
\end{example}

Another instance where Theorem \ref{T:distortion} applies is the following.

\begin{corollary}\label{E:annulus}
Let $X=\B S^1\times [0,1]$ be the closed annulus and let
$h\colon X\to X$ be a homeomorphism preserving the orientation
and the components of the boundary. If the topological rotation numbers
of $h$ restricted to the boundary circles are distinct then
invariant measures supported on boundary circles are Nielsen nonequivalent.
Thus, by Theorem \ref{T:distortion}, $h$ is undistorted in the group of
orientation preserving homeomorphisms of\/ $X$. 
\end{corollary}

This is a corollary of a more general statement (proven in Section~\ref{SS:scc}).

\begin{proposition}\label{P:scc}
Assume that $\F a\in H^1(X,\B Z)$.
Let $\ell_1$  and $\ell_2$ be simple closed curves invariant by a homeomorphism
$g\in\Homeo(X,\F a)$.  Let $\rho_i\in\B R/\B Z$ denote the topological rotation number
of $g$ on $\ell_i$.  If $\rho_1\langle\F a,\ell_1\rangle \neq \rho_2\langle\F a,\ell_2\rangle$
then (any) $g$-invariant measures supported on $\ell_1$ and $\ell_2$ are Nielsen nonequivalent.
\end{proposition}

Observe that the expression from defining $\F K_\alpha$ leads to a
definition of a two-cocycle on the group $\Homeo(X,\F a)$ with trivial coefficients.
Let $g,h\in \Homeo(X,\F a)$ and let $\gamma\colon [0,1]\to X$
be a continuous path from a reference point $x\in X$ to its
image $hx$.  Define
$$
\F G_{x,\alpha}(g,h) := \int_\gamma g^*\alpha-\alpha,
$$
where $\alpha $ is a singular one-cocycle representing the class $\F a$.

In Section \ref{SS:bounded} we consider the case when $\F a\in H^1(X,\B Z)$ and
$\alpha$ is an integer valued one-cocycle representing $\F a$.
In general, the two-cocycle $\F G_{x,\alpha}$ is not a bounded cocycle, but
in Section \ref{SS:bounded} we define
a local rotation number of a homeomorphism $g$ with respect to a point $x$
as $[\F G_{x,\alpha}]\in H^2_{\OP{b}}(\B Z;\B Z)=\B R/\B Z$
provided the cocycle $\F G_{x,\alpha}$ is
a bounded two-cocycle on the cyclic group generated by $g$. 

If $X=\B S^1$ then the two-cocycle $\F G_{x,\alpha}$ corresponding to
the length form is the (well sudied) Euler
cocycle. In particular, the local rotation number 
equals the classical topological rotation number of a
orientation preserving homeomorphism of the circle 
\cite[Section 6.3]{MR1876932}.  More precisely, if $\ell$ is an invariant
circle in $X$ then the local rotation number of a point in $\ell$ equals to
the topological rotation number of the action on $\ell$ times
$\langle\F a,\ell\rangle$.

The points where $\F G_{x,\alpha}$ is a bounded cocycle can be used to detect
nondistortion instead of invariant measures as
the following result shows (which is a consequence of a more general result,
Theorem \ref{T:q-m}, proven in Section~\ref{SS:q-m}).

\begin{theorem}\label{T:rotation}
Let $\F a\in H^1(X;\B Z)$ and let $g\in \Homeo(X,\F a)$
and assume that $X$ is compact.
Let $x$ and $y$ be points such that the cocycles
$\F G_{x,\alpha}$ and $\F G_{y,\alpha}$ are bounded on the cyclic subgroup
generated by $g$. If the local rotation numbers of $g$
at $x$ and $y$ are distinct then $g$ is undistorted
in $\Homeo(X,\F a)$.
\end{theorem}

Notice that Theorem \ref{T:rotation} implies, in particular, Proposition \ref{P:scc}.

\begin{example}
Let $X$ be a closed oriented manifold with non-zero Euler 
characteristic and with positive first Betti number
(e.g. a surface of genus at least two).
Let $F\colon \widetilde{X} \to \B R$ be a function such that
$\OP{d}F = p^*\alpha$, where $p\colon \widetilde{X} \to X$ is
the universal cover and $\alpha$ is a closed one-form
on $X$ with integral periods representing a nonzero
cohomology class $\F a\in H^1(X;\B Z)$.

Let $g\in \Homeo(X,\F a)$ be a homeomorphism and let
$\tilde g\in \Homeo(\widetilde{X})$  denote its lift.
If a point $\tilde x \in \widetilde{X}$ is such that the 
following limit
$$
\lim_{n\to \infty}\frac{F(\tilde g^n(\tilde x))}{n}
$$
exists and is not an integer then $g$ is undistorted
in $\Homeo(X,\F a)$. Indeed, since the Euler characteristic
of $X$ is nonzero the homeomorphism $g$ has a fixed point
$y\in X$. The local rotation number of a fixed point is
equal to zero. On the other hand the local rotation number
of $x:=p(\tilde x)$ is equal modulo integers and up to a sign
to the above limit (see Proposition \ref{P:rotation})
and it follows from the above assumption that
it is nonzero.
\hfill
$\diamondsuit$
\end{example}

\subsection{Distortion in groups}\label{SS:distortion}
Let $\Gamma$ be a finitely generated group.  Define the word norm
associated with fixed set of generators $S$ to be
$$
|g|:=\min\{k\in \B N\,|\,g=s_1\ldots s_k,\,s_i\in S\}.
$$
The {\bf translation length} 
of an element $g\in \Gamma$ is defined to be
$$
\tau(g):=\lim_{n\to \infty}\frac{|g^n|}{n}.
$$
An element $g\in \Gamma$ is called {\bf undistorted} if
its translation length is positive and this property
does not depend on the choice of generators. 
If $G$ is a general (not necessarily finitely generated)
group then $g\in G$ is called {\bf undistorted} if it
is undistorted in every finitely generated subgroup
of $G$. Notice that distortion in a subgroup implies 
distortion in the ambient group.

The distortion is a tool in understanding group actions
on manifolds. For example, it is well known that
certain lattices in semisimple Lie groups contain distorted
elements due to a result of Lubotzky-Mozes and Raghunathan
\cite{MR1828742}. On the other hand, the distortion in
groups of diffeomorphisms of closed manifolds is rare
as shown, for example,  by Franks and Handel \cite{MR2219247}, 
Gambaudo and Ghys \cite{MR2104597}, or Polterovich \cite{MR2003i:53126}. 
This provides restrictions on possible actions of such lattices.

The papers cited above are concerned with the distortion either
in volume preserving or in Hamiltonian diffeomorphisms.
It follows from our results, however, that many elements are undistorted
in groups of homeomorphisms of manifolds of dimension at least two
and with nontrivial first real cohomology. Essentially, this is as much as
one gets for such manifolds. In contrast, Calegari and Freedman 
\cite[Theorem C]{MR2207794} proved that all homeomorphisms of the 
sphere $\B S^n$ are distorted in $\Homeo(\B S^n)$.

\subsection*{Historical remarks}
The cocycle $\F G_{x,\alpha}$ can be defined  for
an arbitrary, not necessarily closed, one-cochain $\alpha$
on a suitably defined subgroup of the group $\Homeo(X)$. It has been first
defined by Is\-ma\-gi\-lov, Losik, and Michor in~\cite{MR2270616} for a
primitive of a symplectic form and further studied by the
authors in \cite{GK2}.

The cocycle $\F K_{\alpha}$ (see Section \ref{SS:one-cocycle} for definition)
appears in Gambaudo and Ghys \cite{MR1452855} and in Arnold and Khesin
\cite[p.~247]{MR1612569} in the case of a symplectic ball.
It has been studied for a general symplectically aspherical
manifold in~\cite{MR2770429}.

The local rotation number generalizes the rotation number 
of a homeomorphism of a circle.  There are related notions 
in the literature.  For example
the rotation vector of a surface diffeomorphism defined by Franks
in \cite[Definition 2.1]{MR1325916}, or
the rotation defined by Burger, Iozzi, and Wienhard
in \cite[Definition 7.1]{MR2680425}.

\section{Proofs of main results}\label{S:preliminaries}
\subsection{The one-cocycle $\F K_{\alpha}$}\label{SS:one-cocycle}
If $g\in G\subset\Homeo(X,\F a)$ then $g^*\alpha - \alpha$
is an exact singular one-cocycle on $X$ and the identity
$\delta(\F K_{\alpha}(g))=g^*\alpha - \alpha$
defines a map
$$
\F K_{\alpha}\colon G\to C^0(X;\B R)/\B R.
$$
It is straightforward to check that $\F K_{\alpha}$ is a one-cocycle
(cf. \cite[Proposition 2.3]{MR2770429}). That is, it satisfies
$$
\F K_{\alpha}(gh)=\F K_{\alpha}(g)\circ h + \F K_{\alpha}(h)
$$
for all $g,h\in G$. 

\begin{lemma}\label{L:continuous}
Assume that $X$ is paracompact.
Let $\F a\in H^1(X;\B R)$. There exists a singular cocycle $\alpha$ representing
the class $\F a$ such that for any homeomorphism $g\in \Homeo(X,\F a)$
the function $\F K_\alpha(g)$ is a continuous function.
\end{lemma}

\begin{remark}
If $X$ is a differentiable manifold and then every real
cohomology class is represented by a smooth and closed differential
form $\alpha$. It follows that for any diffeomorphism
$h\in\Diff(X,\F a)$ the function $\F K_\alpha(h)$ is smooth.
\end{remark}

\begin{proof}[Proof of Lemma \ref{L:continuous}]
Let us consider the real numbers $\B R$ endowed with with the usual order
topology and consider the bundle
$$
\B R\to E=\widetilde X\times_{\pi_1 X}\B R\stackrel{p}\to X.
$$
Since the fibre is contractible and the base is paracompact
it admits a continuous section $s\colon X\to E$. Such a section
defines a continuous equivariant function
$\B a\colon E\to \B R$ by the identity
$p[\tilde x,t]=\B a[\tilde x,t]+sp[\tilde x,t]$.
The equivariance means that
$
\B a[\tilde x,t+s]=\B a[\tilde x,t]+s.
$

Let $X_\F a=\widetilde X\times_{\pi_1 X}\B R^\delta$ be a covering
associated with the class $\F a$, where $\B R^{\delta}$ denotes
the real numbers equipped with the discrete topology. 
Observe that $X_{\F a}$ is equal to
$E$ as a set but it has a finer topology.  
Thus $\B a\colon X_{\F a}\to\B R$ is still a continuous function.

Let  $\tilde g\in \Homeo(X_{\F a})$ be an $\B R$-equivariant lift
of $g\in\Homeo(X,\F a)$. Define a continuous function
$\widehat{\F K}(g)\colon X_{\F a}\to \B R$ by
$$
\widehat{\F K}(g)[\tilde x,t] :=
\B a\left(\tilde g[\tilde x,t]\right)-\B a[\tilde x,t].
$$
Since both $\tilde g$ and $\B a$ are $\B R$-equivariant the function
$\widehat{\F K}(g)$ is $\B R$-invariant and thus descends to a
continuous function $\F K(g)\colon X\to \B R$.

Let us show that $\F K=\F K_\alpha$.  Let $\gamma$ be a path between $x$
and $y$.  Let $\tilde\gamma$ be its lift with endpoints at
$\tilde x$ and $\tilde y$.  Then
\begin{align*}
\F K(g)(y)-\F K(g)(x)
&=\left(\B a(\tilde g\tilde y)-\B a(\tilde g\tilde x)\right)-\left(\B a(\tilde y)-\B a(\tilde x)\right)\\
&=\int_{g\gamma}\alpha-\int_\gamma\alpha=\int_\gamma g^*\alpha-\alpha.
\end{align*}
The second equality above follows from the bijective correspondence 
between singular one-cocycles and $\B R$-equivariant functions
$X_{\F a}\to \B R$ up to the constants. For the convenience
of the reader we explain this folklore fact in Section
\ref{S:singular}.
\end{proof}

\subsection{A seminorm on $\Homeo(X,\F a)$}\label{SS:pseudonorm}
Let $X$ be a compact space.
Let us define a seminorm of an element $g\in\Homeo(X,\F a)$
by
$$
\|g\|_\alpha:=
\sup_{x,y\in X}|\F K_{\alpha}(g)(y) - \F K_{\alpha}(g)(x)|.
$$
This means that $\|\cdot\|_{\alpha}$ is symmetric and
satisfies the triangle inequality.
The finiteness of $\|g\|_{\alpha}$ is a consequence of
the compactness of $X$ according to Lemma \ref{L:continuous}.
It follows that if $\Gamma\subset\Homeo(X,\F a)$ is a subgroup
generated by a finite set $S$ then
$$
C\cdot|g|\geq \|g\|_\alpha,
$$
where $C:=\max \{\|s\|_\alpha| s\in S\}$ and $|g|$ denotes the word norm
of $g\in \Gamma$. This is just a special case of the standard
and straightforward to prove fact that 
any seminorm on a group is Lipschitz 
with respect to the word norm.

\subsection{Proof of Theorem \ref{T:distortion}}\label{SS:proof1}
Let $g$ be homeomorphisms
of $X$ preserving the class $\F a$ and Borel probability measures $\mu$ and $\nu$.
Recall that we need to show that $g$
is undistorted in $\Homeo(X,\F a)$ if $\mu$ and $\nu$ are Nielsen nonequivalent. 

Let $\Gamma\subset \Homeo(X,\F a)$ be an arbitrary finitely
generated group containing $g$. As we mentioned above its inclusion is
Lipschitz with respect to the word metric $|\cdot|$ on $\Gamma$ and
the seminorm $\|\cdot\|_{\alpha}$ on $\Homeo(X.\F a)$.

Observe that the map defined by
$$
\Homeo(X,\F a) \ni h\mapsto \int \F K_\alpha(h)(\mu-\nu)\in\B R
$$
is one-Lipschitz with respect to the seminorm $\|\cdot\|_\alpha$ and 
it is a homomorphism on the cyclic group generated by $g$ (in fact, on a group
of homeomorphisms preserving $\F a$ as well as measures $\mu$ and $\nu$).
From this we get the following estimate of the word norm of $g$.
\begin{align*}
C\frac{|g^n|}n&\geq \frac{\|g\|_{\alpha}}n\\
&\geq \frac 1n\cdot \left|\int \F K_\alpha(g^n)(\mu-\nu)\right|\\
&= \left |\int \F K_\alpha(g)(\mu-\nu)\right|>0
\end{align*}
This shows that the translation length of $g$ in $\Gamma$ is
positive. Since $\Gamma$ is an arbitrary finitely generated
subgroup of $\Homeo(X,\F a)$, this proves that $g$ is undistorted in 
$\Homeo(X,\F a)$.
\qed

\begin{remark}
The above proof is essentially the same as the proof of the Polterovich
theorem \cite[Section 5.3]{MR2003i:53126} about the distortion in the group 
of Hamiltonian diffeomorphisms of a closed symplectically hyperbolic manifold 
presented by the authors in \cite{MR2770429}. 
\end{remark}

\subsection{Proof of Proposition \ref{P:scc}}\label{SS:scc}
Recall that we need to prove that given
simple curve closed curves $\ell_1$ and $\ell_2$ invariant by a homeomorphism
$g\in\Homeo(X,\F a)$ if $\rho_1\langle\F a,\ell_1\rangle \neq \rho_2\langle\F a,\ell_2\rangle$ then
$g$-invariant measures supported on $\ell_i$ and Nielsen nonequivalent.

Let $\alpha$ represent $\F a$.  It is clear that if $\mu_i$ denote any invariant measure
supported on $\ell_i$ then
$$
\int_{\ell_i}\left(\int_x^{gx} \alpha\right)\mu(dx)=\rho_i\int_{\ell_i}\alpha=\rho_i\langle\F a,\ell_i\rangle.
$$
Notice that inner integral depends on the choice of curves between $x$ and $gx$ but once
such choice is made depending continuously on $x$ the value of the integral modulo integers
would not depend on that choice.

Let $\gamma$ be a path between $x$ and $y$, and let $\eta_x$ and $\eta_y$ be paths between
$x$ and $gx$ and $y$ and $gy$ respectively.  Let us choose $\eta_y$ to be a concatenation
of $-\gamma$, $\eta_x$, and $g\gamma$.  Then
$\int_{\eta_x}\alpha+\int_\gamma g^*\alpha-\int_{\eta_y}\alpha-\int_\gamma\alpha=0$.
This can be rewritten as
$$
\F K_\alpha(g)(y)-\F K_\alpha(g)(x) = \int_{\eta_y}\alpha-\int_{\eta_x}\alpha.
$$

Averaging the above equality over $x$ and $y$ with respect to $\mu_1$ and $\mu_2$ respectively
we get
$$
\int \F K_\alpha(g)(\mu_1-\mu_2)=\rho_1\int_{\ell_1}\alpha-\rho_2\int_{\ell_2}\alpha\neq 0.
$$
This proves the claim.
\qed

\begin{corollary}
Let $G\subset \Homeo(X)$ be a group of homeomorphisms
acting trivially on $H^1(X;\B R)$.
Let $g$ be a homeomorphism distorted in $G$.
Let $\ell_1,\ell_2\subset X$ be oriented simple closed curves
preserved by $g$.  Assume also that the classes $[\ell_i]$ are
nonzero in $H^1(X;\B R)$. Let $\rho_1$ and $\rho_2$ be the
topological rotation numbers associated with the action
of $g$ on $\ell_1$ and $\ell_2$ respectively.
Then the following statements hold:
\begin{enumerate}
\item
There exist nonzero integers $k_1,k_2\in \B Z$ such that
$k_1\rho_1 = k_2\rho_2$.
\item
If, moreover, the classes $[\ell_1]$ and $[\ell_2]$ in $H^1(X;\B R)$
are linearly independent then $\rho_1,\rho_2\in \B Q/\B Z$.
\end{enumerate}
\end{corollary}

\begin{proof}
Indeed,
Let $\alpha$ be an one-cocycle with integral periods.

Choosing $\alpha$ such that
$\int_{\ell_i}\alpha \neq 0$ proves the first statement.

To prove the second assertion, we choose $\alpha$ such that
$\int_{\ell_1}\alpha = 0 \neq \int_{\ell_2}\alpha$. It follows
that $\rho_2\cdot \int_{\ell_2}\alpha=0$ and, since
$\int_{\ell_2}\alpha$ is an integer, it implies that
$\rho_2 \in \B Q/\B Z$. The rationality of $\rho_1$ is
proved similarly.
\end{proof}

\section{Further results}\label{S:distortion}
\subsection{The cocycle $\F G_{x,\alpha}$}\label{SS:explicit}
Recall that  $X$ is a path-connected, topological space admitting
a universal cover $\widetilde X\to X$ and
${\F a}\in H^1(X;\B R)$ is a cohomology class represented by
a singular one-cocycle $\alpha$.
Let $x\in X$ be a reference point.
Define a real valued two-cocycle $\F G_{x,\alpha}$
on the group $\Homeo(X,\F a)$ of homeomorphisms of
$X$ preserving the class $\F a$ by the following
formula
$$
\F G_{x,\alpha}(g,h):=\int_\gamma g^*\alpha - \alpha
$$
where $\gamma$ is a path from $x$ to $hx$.

\begin{lemma}\label{L:K-G}
\  
\begin{enumerate}
\item\label{i:K-G}
If $h$ and $g$ are homeomorphisms preserving $\F a=[\alpha]$ then
$$
\F G_{x,\alpha}(g,h)=\F K_\alpha(g)(hx)-\F K_\alpha(g)(x).
$$
\item
The value $\F G_{x,\alpha}(g,h)$ does not depend on the
choice of a path from $x$ to $hx$.
\item
The function $\F G_{x,\alpha}$ is a two-cocycle on $\Homeo(X,\F a)$.
That is it satisfies the following identity:
$$
\F G_{x,\alpha}(h,k)-\F G_{x,\alpha}(gh,k)+
\F G_{x,\alpha}(g,hk)-\F G_{x,\alpha}(g,h)=0.
$$
\item
The cohomology class of the cocycle $\F G_{x,\alpha}$
depends neither on the choice of the reference
point $x$ nor on the choice of the cocycle~$\alpha$ (only on the
cohomology class $\F a$).
\item If either $g$ preserves $\alpha$ or $h$ preserves $x$
then $\F G_{x,\alpha}(g,h)=0$.
\qed
\end{enumerate}
\end{lemma}

\begin{proof}
For the sake of consistency we prove part \ref{i:K-G} of the lemma,
leaving the other, straightforward items, which will not be used
in the paper, to the reader.

It is an immediate consequence the definition of the 
cocycle~$\F K_{\alpha}$. Indeed, we have
$$
\F G_{x,\alpha}(g,h)=\int_{x}^{hx}g^*\alpha - \alpha 
=\int_{x}^{hx}\delta(\F K_{\alpha}(g))
=\F K_\alpha(g)(hx)-\F K_\alpha(g)(x).
$$
\end{proof}

In what follows, as the one-cocycle $\alpha$ is fixed in this section,
we would write $\F G_x$ instead $\F G_{x,\alpha}$ for short.

\subsection{Quasimorphisms}\label{SS:q-m}
Let $\F q\colon G\to \B R$ be a map defined on a group $G$.
The {\bf defect} $D(\F q)$ of the map $\F q$ is defined to be
$$
D(\F q):=\sup_{g,h\in G}|\F q(g) - \F q(gh) + \F q(h)|.
$$
If the defect of $\F q$ is finite then $\F q$ is called
a {\bf quasimorphism}. A quasimorphism $\F q$ is called
{\bf homogeneous} if $\F q(g^n)=n\F q(g)$ for all $n\in \B Z$
and $g\in G$. For every quasimorphism $\F q$ the formula
$$
\widehat {\F q}(g):=\lim_{n\to\infty}\frac{\F q(g^n)}{n}
$$
defines a homogeneous quasimorphism called the homogenisation
of~$\F q$. Moreover, $|\widehat {\F q}(g)-\F q(g)|\leq D$
for all $g\in G$ \cite[Lemma 2.21]{MR2527432}. Thus $\F q$ is unbounded
if and only if so is its homogenisation. 

\begin{proposition}\label{P:q-m}
Let $\F a\in H^1(X;\B R)$.
Let $G\subseteq \Homeo(X,\F a)$ be a subgroup on which
the cocycles $\F G_x$ and $\F G_{y}$ are bounded,
for some $x,y\in X$.
Then the map
$\F q\colon G\to \B R$ defined by
$$
\F q(g):= \F K_{\alpha}(g)(y) - \F K_{\alpha}(g)(x)
$$
is a quasimorphism on $G$ and
$D(\F q)\leq\|\F G_x-\F G_y\|\leq\|\F G_x\|+\|\F G_y\|$,
where $\|\cdot\|$ denotes the supremum norm of a bounded
function. 
\end{proposition}

\begin{proof}
This is a straightforward computation using the cocycle identity for $\F K_{\alpha}$.
\begin{align*}
\F q(f) - \F q(fg) + \F q(g) =&
\F K_{\alpha}(f)(y) - \F K_{\alpha}(f)(x)\\
&-(\F K_{\alpha}(f)(gy) + \F K_{\alpha}(g)(y)  - \F K_{\alpha}(f)(gx) - \F K_{\alpha}(g)(x))\\
&+\F K_{\alpha}(g)(y) - \F K_{\alpha}(g)(x)\\
=&\F K_{\alpha}(f)(gx) - \F K_{\alpha}(f)(x) - (\F K_{\alpha}(f)(gy) - \F K_{\alpha}(f)(y))\\
=&\F G_x(f,h) - \F G_y(f,g).
\end{align*}
\end{proof}

\begin{example}\label{E:torus}
In this example we show that the boundedness of $\F G_x$ depends
on the choice of a point $x\in X$.
Let $X=\B R/\B Z\times \B R\cup\{\infty\}$ be the two-dimensional torus.
Let $\alpha$ be a singular one-cocycle defined by
$$
\int_{\gamma}\alpha := \tilde{\gamma}(1)-\tilde{\gamma}(0),
$$
where $\tilde{\gamma}\colon [0,1]\to \B R$ is a lift of
the composition of $\gamma$ followed by the projection
onto $\B R/\B Z$. Let $\F a$ be the class of $\alpha$.

Let $g\in \Homeo(X,\F a)$ be a homeomorphism defined by
$$
g(t,x):= (t+|x+1|-|x|,x+1).
$$
Then $\F K_{\alpha}(g^n)(t,x) = |x+n| - |x|$ and it follows that
\begin{align*}
\F G_{(0,0)}(g^m,g^n) &=
\F K_{\alpha}(g^m)(g^n(0,0))-\F K_{\alpha}(g^m)(0,0)\\
&= |m+n| - |n| -|m|.
\end{align*}
This shows that $\F G_{(0,0)}$ is unbounded (in fact, the
cocycle $\F G_{(t,x)}$ is un\-boun\-ded whenever $x$ is finite).
On the other hand, $g$ acts trivially on the
circle $\B R/\B Z\times \{\infty\}$ and hence
$\F G_{(t,\infty)}=0$.
\hfill $\diamondsuit$
\end{example}

\begin{example}
If $g$ is a time-one map of a gradient flow then $\F G_{x,\alpha}$
is bounded at every $x$ and the local rotation number of $g$ 
is equal to zero. 
\hfill
$\diamondsuit$
\end{example}

\begin{theorem}\label{T:q-m}
Let $\F a\in H^1(X;\B R)$ and let $g\in \Homeo(X,\F a)$
and assume that $X$ is compact.
Suppose that for some points $x,y \in X$ the cocycles
$\F G_x$ and $\F G_y$ are bounded on the cyclic subgroup
$\langle g\rangle \subset \Homeo(X,\F a)$ ge\-ne\-rated by $g$.
If the above quasimorphism $\F q\colon \langle g\rangle \to \B R$ 
is unbounded then $g$ is undistorted in $\Homeo(X,\F a)$.
\end{theorem}

\begin{remark}
It is often the case that to prove that an element $g$ is 
undistorted in a group $G$ one constructs a
homogeneous quasimorphism $\F q\colon G\to \B R$ such that
$\F q(g)\neq 0$. Constructing such a quasimorphism is
in general very difficult. The advantage of the above theorem
is that we only need to check that a naturally defined
quasimorphism on a cyclic group is unbounded.
\end{remark}

\begin{proof}[Proof of Theorem \ref{T:q-m}]
Let $\Gamma$ be a finitely generated subgroup of $\Homeo(X,\F a)$
containing $g$.
Let $\widehat{\F q}\colon \langle g\rangle\to \B R$
be the homogenisation of the quasimorphism $\F q$.
The following calculation of the translation length of $g$
shows that $g$ is undistorted in $\Gamma$.
\begin{align*}
C\cdot\tau(g)& = \lim_{n\to \infty}\frac{C\cdot|g^n|}{n}\\
& \geq \lim_{n\to \infty}\frac{\|g^n\|_{\alpha}}{n}\\
& \geq \lim_{n\to \infty}\frac{|\F q(g^n)|}{n}\\
& = |\widehat{\F q}(g)| > 0.
\end{align*}
Since $\Gamma$ is arbitrary,
the element $g$ is undistorted in $\Homeo(X,\F a)$.
\end{proof}

Notice that Theorem \ref{T:q-m} also implies Corollary \ref{C:distortion}.
Recall that we need to prove that if $x$ and $y$ are fixed points of $g$
and $\int_\gamma g^*\alpha-\alpha\neq 0$ then $g$ is undistorted in
$\Homeo(X,\F a)$.

First, observe that the cocycles $\F G_x$ and $\F G_y$
vanish identically on the cyclic group $\langle g \rangle$ because 
$x$ and $y$ are fixed points of $g$. By Proposition \ref{P:q-m} 
the defect of $\F q$ is zero (since it is bounded by $\|\F G_x\|+\|\F G_y\|=0$) 
and we obtain that $\F q\colon \langle g\rangle \to \B R$ 
is a homomorphism of groups. Furthermore
\begin{equation*}
\F q(g)=\F K_{\alpha}(g)(y) - \F K_{\alpha}(g)(x)= \int_{\gamma}g^*\alpha -\alpha \neq 0
\end{equation*}
according to the hypothesis. Therefore $\F q$ is unbounded and
the statement follows from Theorem \ref{T:q-m}.

\subsection{Local rotation number}\label{SS:bounded}
In what follows we are interested in bounded cohomology
of a group with the integer coefficients; see Gromov
\cite{MR686042} and Monod \cite{MR1840942} for a background
on bounded cohomology.

We assume that $\F a\in H^1(X;\B Z)$ and $\alpha$ is an integer-valued
one-cocycle on $X$.  Therefore $\F G_{x,\alpha}$ is an integer-valued
two-cocycle on $\Homeo(X,\F a)$.

\begin{example}
(Ghys \cite[Section 6.3]{MR1876932})
\label{E:boundedZ}
The second bounded cohomology $H^2_{\OP{b}}(\B Z;\B Z)$
of the integers with integer coefficients is isomorphic
to $\B R/\B Z$. To see this let
$\F c \colon \B Z \times \B Z\to \B Z$
be a bounded two-cocycle. As an ordinary cocycle it is
a coboundary since the second group cohomology of the group
of integers is trivial.
If $\F c=\delta\F b$ then, since $\F c$ is
bounded, the cochain $\F b$ is a quasimorphism.
The homogenisation of $\F b$ (which is a real cochain in general)
is given by $\widehat {\F b}(n)=rn$ for some real number $r\in \B R$.
The required isomorphism
$$
H^2_{\OP{b}}(\B Z;\B Z)\to \B R/\B Z
$$
is defined by $[\F c]\mapsto r +\B Z$.
\hfill $\diamondsuit$
\end{example}

Let $g\in \Homeo(X,\F a)$ and let $x\in X$ be a point for which the
cocycle $\F G_{x,\alpha}$ is bounded on the cyclic group generated by $g$.
The cohomology class
$$
\OP{\bf rot}_{x,\alpha}(g)=[\F G_{x,\alpha}]\in H_{\OP{b}}^2(\langle g\rangle;\B Z)=\B R/\B Z
$$ 
is called the {\bf local rotation number} of $g$ at the point $x\in X$.

Let us explain the geometry of the local rotation number.
Take a path $\eta_{x,1}\colon [0,1]\to~X$ from $x$ to $gx$
and let $\eta_{x,n}$ be the concatenation of paths
$g^k(\eta_{x,1})$ for $k$ ranging from $0$ to $n-1$.
Define a map $\F b_x\colon \langle g\rangle \to \B R$
by
\begin{equation}\label{Eq:b}
\F b_x(g^n):=-\int_{\eta_{x,n}}\alpha.
\end{equation}
Observe that $\delta\F b_x = \F G_x$ on the cyclic group
$\langle g\rangle$. Since $\F G_x$ is bounded on $\langle g\rangle$
we get that $\F b_x$ is a quasimorphism and that its
homogenisation satisfies $\widehat {\F b}_x(g^n)=r_x(g)n$ for
a suitable representative of the local rotation number of $g$ at $x$.
This shows that there exists a constant $C_x>0$ such that
$$
|\F b_x(g^n) - r_x(g)n|\leq C_x  
$$
for all $n\in \B Z$. We thus obtain that
\begin{equation}\label{Eq:r}
\lim_{n\to \infty} \frac{\F b_x(g^n)}{n} = r_x(g)
\end{equation}
and hence the fractional part of above limit represents the local rotation
number of $g$ at $x$:
$$
\OP{\bf rot}_{x,\alpha}(g)=r_x(g)+\B Z.
$$
Indeed, since $\alpha$ has integral
periods, the dependence of $\F b_x(g^n)$ on the choice of
the path $\eta_{x,1}$ is up to an integer constant only.
This implies that the above computation of the
local rotation number does not depend on the choice
of a path $\eta_{x,1}$. The next result immediately follows
from the above discussion.

\begin{proposition}\label{P:rotation}
Let $X$ be a smooth compact manifold and let $p\colon X_{\F a}\to X$
be the cyclic covering associated with $\F a\in H^1(X;\B Z)$.
Assume that $\F a$ is represented by a closed smooth one-form $\alpha$.
Let $F\colon X_{\F a}\to \B R$ be a smooth function such that
$\OP{d}F=p^*\alpha$. Then
$$
\OP{\bf rot}_{x,\alpha}(g) = -\lim_{n\to \infty}\frac{F(\tilde g^n(\tilde x))}{n} + \B Z
$$
provided that the limit exists.\qed
\end{proposition}

\subsection{Proof of Theorem \ref{T:rotation}}\label{SS:proof_rotation}
In order to apply Theorem \ref{T:q-m} we need to prove that the quasimorphism 
$\F q\colon \langle g\rangle \to \B R$ from Proposition \ref{P:q-m} is unbounded.

Let $\gamma,\eta_{x,n},\eta_{y,n}\colon [0,1]\to X$
be  paths from $x$ to $y$, $x$ to $g^nx$ and $y$
to $g^ny$ respectively and $n\in \B Z$. As above assume that
$\eta_{x,n}$ is a concatenation of the paths $g^k (\eta_{x,1})$
for $k$ ranging from $0$ to $n-1$ and similarly for $\eta_{y,n}$.

Let $\F b_x,\F b_y\colon G\to \B R$ and $r_x(g), r_y(g)\in \B R$ be
defined as in \ref{Eq:b} and \ref{Eq:r}.
Let $\square_n$ be a concatenation of $-\gamma$, $\eta_{x,n}$,
$g^n \gamma$ and $-\eta_{y,n}$. We get the following computation.
\begin{align*}
\F q(g^n) &= \int_{\gamma} (g^n)^*\alpha - \alpha\\
&= \int_{\square_n}\alpha - \int_{\eta_{x,n}}\alpha + \int_{\eta_{y,n}}\alpha\\
&= n \int_{\square_1}\alpha + \F b_x(g^n) - \F b_y(g^n)\\
&= n \left( \int_{\square_1}\alpha + (r_x(g)-r_y(g))\right ) + O(1).
\end{align*}
Since $\alpha$ has integral periods and the difference $r_x(g)-r_y(g)$ is not an integer
by the hypothesis, we get that the quasimorphism $\F q$ is unbounded. 
\qed

\section{Appendix: On singular one-cocycles}\label{S:singular}
The results of this section are used to prove Lemma \ref{L:continuous}.

Let $\B A$ be an Abelian group which is a trivial
coefficient system over $X$.
Since $H^1(X;\B A)=\Hom(\pi_1(X),\B A)$ one can define a cover
$$
\B A\to X_\F a:=\widetilde X\times_{\pi_1(x)}\B A\to X,
$$
where $\widetilde X$ is the universal cover of $X$ and $\pi_1(x)$
acts on $\B A$ via homomorphism defined by $\F a$.
In what follows, the action $\B A\times X_{\F a}\to X_{\F a}$ by 
the deck transformations will be denoted additively: $(a,z)\mapsto a+z$.

Let $x\in X$ be a reference point in $X$ and let
$\tilde x\in p^{-1}(x)$ be a reference point
in $X_{\F a}$.
Let $\alpha$ be a singular cocycle representing the class $\F a$.
That is, $\alpha$ is a homomorphism $C_1(X;\B A)\to \B A$
defined on the group of chains on $X$ with the coefficients
in $\B A$. It defines an $\B A$-equivariant map
$\B a\colon X_\F a\to\B A$ in the following way.
Given a point $\tilde y\in p^{-1}(y)$ let $\gamma\colon [0,1]\to X$
be a path from $x$ to $y$. Let $\tilde\gamma\colon [0,1]\to X_{\F a}$
be its lift such that $\tilde\gamma(0)=\tilde x$.  Then
we define $\B a\left(\tilde y\right)$ as the unique element such that
$\int_\gamma\alpha+\tilde y
=\B a\left(\tilde y\right)+\tilde\gamma(1)$.
If we put $\tilde y := \tilde\gamma(1)$ we obtain that
$$
\B a(\tilde\gamma(1)) = \int_\gamma\alpha.
$$

Let us check that $\B a$ does not depend on the choice
of the path $\gamma$. Let $\gamma_\pm$ be two paths
from $x$ to $y$ and let $\B a_-$ and $\B a_+$
denote the corresponding maps.
By letting $\tilde y = \tilde\gamma_+(1)$ in the
equality
$$
\int_{\gamma_+}\alpha +\tilde\gamma_-(1)=
\int_{\gamma_-}\alpha +\tilde\gamma_+(1)
$$
we get
$$
\int_{\gamma_+}\alpha +\tilde\gamma_-(1)=
\int_{\gamma_-}\alpha +\tilde y
$$
which shows that
$\B a_+(\tilde y)=\int_{\gamma_+}\alpha = \B a_-(\tilde y)$
as claimed.

The equivariance of $\B a$ is immediate from the definition.
Another choice of a reference point results in changing $\B a$
by an additive constant.

Let $\B a\colon X_\F a\to\B A$ be an $\B A$-equivariant function.
Let $\gamma\colon [0,1]\to X$ be a path and let
$\tilde\gamma\colon [0,1]\to X_a$ be its lift. The following
formula defines a singular one-cocycle with values in $\B A$.
$$
\int_\gamma\alpha=
\B a\left(\tilde\gamma(1)\right)-
\B a\left(\tilde\gamma(0)\right)
$$

\begin{lemma}\label{L:singular}
The above constructions are inverse to each other and
hence provide a bijective correspondence between
singular one-cocycles in the class $\F a\in H^1(X,\B A)$
and $\B A$-equivariant maps $\B a\colon X_\F a\to \B A$
up to the constants.
\end{lemma}

\begin{proof}
Let $\alpha $ be a singular one-cocycle representing the class $\F a$.
It defines an equivariant map $\B a\colon X_{\F a}\to \B A$
such that $\int_\gamma\alpha+\tilde y
=\B a\left(\tilde y\right)+\tilde\gamma(1)$
for every path $\gamma\colon [0,1]\to X$ from $x$ to $y$.
We need to check that
$
\int_\gamma\alpha=
\B a\left(\tilde\gamma(1)\right)-
\B a\left(\tilde\gamma(0)\right)
$.

Let $\tilde y:=\tilde\gamma(1)$ where the lift
$\tilde\gamma$ is chosen so that $\B a(\tilde\gamma(0))=0$.
Then
$$
\int_\gamma\alpha+\tilde\gamma(1)=
\B a\left(\tilde\gamma(1)\right)+\tilde\gamma(1)
$$
implies that $
\int_\gamma\alpha=\B a\left(\tilde\gamma(1)\right)$.

Conversely, let $\B a\colon X_{\F a}\to \B A$ be an $\B A$-equivariant
map. It defines a singular cocycle $\alpha$ by the identity
$\int_\gamma\alpha = \B a(\tilde\gamma(1))$, where
$\tilde\gamma$ is a lift of $\gamma$ such that
$\B a(\tilde\gamma(0))=0$. We then clearly get
that $\int_\gamma\alpha+\tilde\gamma(1)=
\B a\left(\tilde\gamma(1)\right)+\tilde\gamma(1)$.
\end{proof}

\subsection*{Acknowledgements}
The authors thank Alessandra Iozzi, Assaf Libman, Aleksy Tral\-le,
and anonymous referee for helpful comments and discussions.

\'S.R. Gal is partially supported by Polish MNiSW grant N N201 541738
and Swiss NSF Sinergia Grant CRSI22-130435.

\bibliography{../../bib/bibliography}
\bibliographystyle{acm}

\end{document}